\documentclass{amsart}
\usepackage{amsfonts}
\usepackage{amstext}
\usepackage{amsmath,amscd,tensor}
\usepackage{amssymb}
\usepackage[mathscr]{eucal}
\usepackage{enumitem}
\usepackage{url}
\usepackage{color}
\usepackage{booktabs,array,multirow}
\usepackage{mathtools}
\usepackage{xhfill}
\usepackage{pdflscape}
\usepackage{graphicx}
\usepackage{verbatim}
\usepackage{wrapfig}
\usepackage{titlesec}

\theoremstyle{plain}
\newtheorem{theorem}{Theorem}[section]
\newtheorem{lemma}[theorem]{Lemma}
\newtheorem{proposition}[theorem]{Proposition}
\newtheorem{corollary}[theorem]{Corollary}
\newtheorem{assumption}[theorem]{Standing Assumption}
\theoremstyle{definition}
\newtheorem{definition}[theorem]{Definition}
\newtheorem{remark}[theorem]{Remark}
\newtheorem{example}[theorem]{Example}

\newcommand{\R}{\mathbb{R}}
\newcommand{\Z}{\mathbb{Z}}
\newcommand{\N}{\mathbb{N}}

\newcommand{\B}{\mathcal{B}}

\DeclareMathOperator{\Var}{Var}
\def\CMM{{\mathcal C\mathcal M\mathcal M}}

\renewcommand{\L}{\mathcal{L}}

\renewcommand{\dot}{\bullet}
\newcommand{\taboo}[4]{\tensor*[_{#2}]{#1}{_{#3}^{#4}}}
\newcolumntype{M}[1]{>{\centering}m{#1\textwidth}}

\titleformat{\section}[hang]{\sc\filcenter}{\thesection.}{1em}{}
\titleformat{\subsection}[hang]{\bfseries\filcenter}{\thesubsection.}{1em}{}

\begin{document}

\title[Infimum of Lipschitz Constants]{The Infimum of Lipschitz Constants in the Conjugacy Class of an Interval Map}
\author{Jozef Bobok}
\address{Czech Technical University in Prague, FCE, Th\'{a}kurova 7, 166 29 Praha 6, Czech Republic}
\email{jozef.bobok@cvut.cz}
\urladdr{http://mat.fsv.cvut.cz/bobok/}
\author{Samuel Roth}
\address{Silesian University in Opava, Na Rybni\v{c}ku 626/1, 746 01 Opava, Czech Republic}
\email{samuel.roth@math.slu.cz}


\begin{abstract}
How can we interpret the infimum of Lipschitz constants in a conjugacy class of interval maps?  For positive entropy maps, the exponential of the topological entropy gives a well-known lower bound.  We show that for piecewise monotone interval maps as well as for $C^{\infty}$ interval maps, these two quantities are equal, but for countably piecewise monotone maps, the inequality can be strict.  Moreover, in the topologically mixing and Markov case, we characterize the infimum of Lipschitz constants as the exponential of the Salama entropy of a certain reverse Markov chain associated with the map.  Dynamically, this number represents the exponential growth rate of the number of iterated preimages of nearly any point.
\end{abstract}


\maketitle


\section{Introduction}\label{sec:intro}

There is a long history in interval dynamics relating Lipschitz constants with topological entropy.  Under suitable assumptions on the continuous map $f:[0,1]\to[0,1]$, and letting $\lambda=\exp h(f)$ denote the exponential of the topological entropy of $f$, there are constructions for
\begin{itemize}
\item a conjugate map\footnote{In fact, $g$ is piecewise linear with slope $\pm\lambda$ on each piece.} $g$ with Lipschitz constant $\lambda$,\\ \strut \hfill - Parry~\cite{P}, for $f$ piecewise monotone and transitive.
\item a factor map\footnote{Again with ``constant slope'' $\pm\lambda$.  By sacrificing transitivity, Milnor and Thurston obtain a (nondecreasing) semiconjugacy in place of Parry's conjugacy.} $g$ with Lipschitz constant $\lambda$,\\ \strut \hfill - Milnor-Thurston~\cite{MT}, for $f$ piecewise monotone, $h(f)>0$,
\item an extension map\footnote{Here, $g$ need not be piecewise linear.  The extension is nondecreasing in the sense that $\phi\circ g=f \circ\phi$ for some nondecreasing map $\phi:[0,1]\to[0,1]$} $g$ with Lipschitz constant $\lambda$,\\ \strut \hfill - Bobok~\cite{B}, for $f$ mixing, $h(f)<\infty$, and $f(0,1)\supset \{0,1\}$.
\end{itemize}
In each construction $g$ inherits the same entropy as $f$, and we may say that its Lipschitz constant is best possible, in the sense that a positive-entropy interval map cannot have a Lipschitz constant smaller than the exponential of its entropy.

  Motivated by the above results, we propose to study the following natural conjugacy invariant for continuous interval maps:
\begin{multline}\label{mult:inf}
\Lambda(f) = \inf \{ \lambda : \textnormal{for some homeomorphism } \psi:[0,1]\to[0,1], \\ \psi\circ f\circ\psi^{-1} \textnormal{ has Lipschitz constant } \lambda \}.
\end{multline}

We work with three natural classes of maps: piecewise monotone, $C^{\infty}$, and countably piecewise monotone. For piecewise monotone maps as well as for $C^{\infty}$ maps we find $\Lambda(f)=\exp h(f)$ - Corollaries~\ref{cor:pm} and~\ref{cor:smooth} - but in the absence of transitivity the infimum need not be attained.  For countably piecewise monotone maps, we assume topological mixing and the presence of a countable Markov partition. Then $\Lambda(f)$ is characterized as the growth rate  $\limsup_{n\to\infty} \sqrt[n]{\#f^{-n}(\{x\})}$ of the number of preimages of an arbitrary point $x$, provided $x$ is not an accumulation point of the Markov partition set - Theorem \ref{th:preimages}. We show by example that this number may be strictly larger than the entropy - Example~\ref{ex:gap}. Nevertheless, if $f$ is locally eventually onto or $C^1$ smooth, then once again $\Lambda(f)=\exp h(f)$ - Theorem \ref{th:leo}.

\section{Preliminary observations}

The first observation to record is that topological entropy provides a natural lower bound for Lipschitz constants.

\begin{proposition}[\cite{KH}]\label{prop:lipbound}
Suppose $f:[0,1]\to[0,1]$ has Lipschitz constant $\lambda$ and topological entropy $h(f)$.  Then
\begin{equation*}
h(f)\leq\max\{0,\log\lambda\}.
\end{equation*}
\end{proposition}

\begin{assumption}\label{ass}
We ignore interval maps $f:[0,1]\to[0,1]$ for which $\cap_{i=0}^\infty f^i([0,1])$ is a singleton.
\end{assumption}

Under this assumption, the following facts hold.
\begin{itemize}
\item $f$ cannot have Lipschitz constant $\lambda<1$.  Otherwise it would be a uniform contraction, with a unique fixed point attracting the entire interval.
\item Consequently, the conclusion of Proposition~\ref{prop:lipbound} simplifies to $h(f)\leq\log\lambda$.
\end{itemize}

\begin{corollary}\label{cor:lipbound}
Under our standing assumption, $\Lambda(f)\geq \exp h(f)$.
\end{corollary}

\begin{example}
The maps $f(x)=x^2$ and $g(x)=\sqrt{x}$, $x\in [0,1]$, are in the same conjugacy class, $g$ is not Lipschitz. At the same time  $\Lambda(f)=\Lambda(g)=1$ is not attained. The conjugate maps with Lipschitz constants close to 1 are $g_t(x)=x^t$, $t>1$.  Explicitly, $\psi_t\circ f = g_t\circ \psi_t$, where $\psi_t(x)=e^{-(\ln 1/x)^{\log_2 t}}$ (with values at the endpoints given by continuity).
\end{example}

\begin{example}
Let $g:[0,1]\to [0,1]$ be a continuous cocountably $\infty$-fold map, i.e., a map for which $\#g^{-1}(y)=\infty$ for all but countably many points $y$ in $[0,1]$ (for such a map see \cite[Section 5]{Bo05}). For $a,b\in (0,1)$ and a positive $\varepsilon$ satisfying $b+\varepsilon<a<a+\varepsilon<1$ consider a continuous map $f:[0,1]\to [0,1]$ defined as: $f(x)=b+\varepsilon g(\frac{1}{\varepsilon}(x-a))$ for $x$ in $[a,a+\varepsilon]$, with $f$ affine on the intervals $[0,a]$ and $[a+\varepsilon,1]$. Then for any homeomorphism $\psi$ of the interval $[0,1]$, the map $h_{\psi}=\psi\circ f\circ\psi^{-1}$ is cocountably $\infty$-fold with respect to $\psi([b,b+\varepsilon])$, i.e., $\#h_{\psi}^{-1}(y)=\infty$ for all but countably many points $y$ in $\psi([b,b+\varepsilon])$. Since any Lipschitz interval map has (Lebesgue) almost all preimage sets finite \cite{Ce1}, it means that $h_{\psi}$ is not Lipschitz  for any $\psi$ hence $\Lambda(f)=\infty$. The equality $h(f)=0$ is clear.
\end{example}


\section{When growth rate of variation equals topological entropy}\label{sec:pm}

This section addresses two natural classes of interval maps -- piecewise monotone maps and $C^\infty$ maps -- using the tool of total variation.  We do not require transitivity of our maps, nor do we assume any Markov conditions.

\begin{definition}
The \emph{total variation} of a real-valued function $f:[a,b]\to\R$ is
\begin{equation*}
\Var f := \sup \sum_{i=0}^{s-1}|f(p_{i+1})-f(p_i)|,
\end{equation*}
where the supremum is taken over all finite sequences $p_0 < p_1 < \cdots < p_s$ of elements of the domain $[a,b]$.
\end{definition}

\begin{lemma}\label{lem:var}
Let $f:[0,1]\to\mathbb R$ be continuous, $J\subset[0,1]$ an interval, and $n\in\N$.  Then $Var f^n|_{f(J)} \leq Var f^{n+1}|_{J}$.
\end{lemma}

\begin{proof}
This is a standard exercise in real analysis.
\end{proof}

Combining the notions of total variation and iteration, we obtain

\begin{definition}
The \emph{growth rate of variation} of an interval map $f$ is defined to be
\begin{equation*}
\nu(f):=\limsup_{n\to\infty} \sqrt[n]{\Var f^n}.
\end{equation*}
\end{definition}




Now we present the main results of this section.

\begin{theorem}\label{th:pm}
Let $f$ be a continuous interval map.  Then $\Lambda(f)\leq\nu(f)$.
\end{theorem}

\begin{corollary}\label{cor:pm}
If $f$ is piecewise monotone, then $\Lambda(f)=\exp h(f)$.
\end{corollary}

\begin{corollary}\label{cor:smooth}
If $f$ is $C^\infty$ smooth, then $\Lambda(f)=\exp h(f)$.
\end{corollary}

\begin{proof}[Proof of corollaries]
If $f$ is piecewise monotone or $C^\infty$ smooth, then the entropy is given by
$
h(f)=\max\{0,\log\nu(f)\}.
$
For piecewise monotone maps this result is due to Misiurewicz and Szlenk \cite{ALM, MS}.  For $C^\infty$ maps it follows from the work of Yomdin \cite{Y}.

Standing assumption~\ref{ass} gives $\nu(f)\geq1$, since $\Var(f^n)$ is bounded away from zero by the length of the interval $\cap_{i=0}^\infty f^i([0,1])$.  Now apply Corollary~\ref{cor:lipbound} and Theorem~\ref{th:pm}.
\end{proof}

\begin{proof}[Proof of Theorem~\ref{th:pm}]
We may assume $\nu(f)$ is finite, or there is nothing to prove. We construct interval maps conjugate to $f$ with Lipschitz constants arbitrarily close to $\nu(f)$. Let
\begin{equation}\label{phi}
\phi(x):=\sum_{n=0}^\infty \frac{\Var f^n|_{[0,x]}}{(\nu(f)+\epsilon)^n}.
\end{equation}
We claim that $\phi(1)<\infty$, $\phi$ is a homeomorphism onto its image, and $\phi \circ f \circ \phi^{-1}$ has Lipschitz constant $\nu(f)+\epsilon$.

By the definition of $\nu$, we have $\Var f^n \leq (\nu(f)+\frac{\epsilon}{2})^n$ for all $n$ greater than or equal to some $N\in\N$.  Applying the comparison test with the geometric series
\begin{equation}\label{comp}
\sum_{n=N}^\infty \frac{(\nu(f)+\frac{\epsilon}{2})^n}{(\nu(f)+\epsilon)^n},
\end{equation}
we conclude that $\phi(1)$ is finite.

We have a trivial inequality $\Var f^n|_{[0,x]} \leq \Var f^n$ for all $x\in[0,1]$.  Therefore we can use the same geometric series~\eqref{comp} in the Weierstrass M-test to conclude that the series in~\eqref{phi} converges uniformly.  Then $\phi$, being a uniform limit of continuous functions, is continuous.

Since $f^0$ is the identity function, we have $\Var f^0|_{[0,x]}=x$, and we may write $\phi(x)=x+\sum_{n=1}^\infty \frac{\Var f^n|_{[0,x]}}{(\nu(f)+\epsilon)^n}$.  Now it is clear that $\phi$ is strictly monotone increasing.

The last three paragraphs combined show that $\phi:[0,1]\to[0,\phi(1)]$ is a homeomorphism onto its image.  Therefore we can form the composition $g=\phi\circ f\circ\phi^{-1}$.

If $x,y\in[0,1]$, we will use the notation $[x;y]$ for the interval with endpoints $x$ and $y$ regardless of their ordering.  By~\eqref{phi} we have
\begin{align*}
&\Big|g(\phi(y)) - g(\phi(x))\Big| =
\Big|\phi(f(y))-\phi(f(x))\Big| =\\
&=\sum_{n=0}^\infty \frac{Var f^n|_{[f(x);f(y)]}}{(\nu(f)+\epsilon)^n} \leq \sum_{n=0}^\infty \frac{Var f^n|_{f([x;y])}}{(\nu(f)+\epsilon)^n}\cdots,
\end{align*}
which by Lemma~\ref{lem:var} becomes
\begin{equation*}
\cdots  \leq
\sum_{n=0}^\infty \frac{Var f^{n+1}|_{[x;y]}}{(\nu(f)+\epsilon)^n} =
(\nu(f)+\epsilon) \sum_{n=1}^\infty \frac{Var f^n|_{[x;y]}}{(\nu(f)+\epsilon)^n} \leq \cdots
\end{equation*}
Adding back a term with index zero and again applying~\eqref{phi}, we get
\begin{equation*}
\cdots \leq (\nu(f)+\epsilon) \sum_{n=0}^\infty \frac{Var f^n|_{[x;y]}}{(\nu(f)+\epsilon)^n} = (\nu(f)+\epsilon) \Big|\phi(y)-\phi(x)\Big|.
\end{equation*}
Since $\phi:[0,1]\to[0,\phi(1)]$ is surjective, this shows that $g$ has Lipschitz constant $\nu(f)+\epsilon$.  Finally, by rescaling $g$, (conjugating by a linear homeomorphism $[0,\phi(1)]\to[0,1]$), we get a map on the unit interval $[0,1]$ conjugate to $f$ with Lipschitz constant $\nu(f)+\epsilon$.

\end{proof}

\section{Topologically mixing Maps with a Countable Markov Partition}\label{sec:cpmm}

We turn our attention now to continuous interval maps which are topologically mixing and countably Markov.

\emph{Topological mixing} of $f$ means that for each pair of nonempty open sets $U,V$, there is $n_0\geq0$ such that $f^n(U)\cap V\neq\emptyset$ for every $n\ge n_0$.  For topologically mixing interval maps, our Standing Assumption~\ref{ass} is redundant and may be dropped.

%

An interval map $f$ is said to be \emph{countably Markov} if there is a closed and countable (or finite) set $P$, $0,1\in P$, which is forward invariant $f(P)\subset P$, and such that $f|_I$ is monotone on each component $I$ of $[0,1]\setminus P$.  Such a set $P$ will be called a \emph{partition set} for $f$.  We denote by $\B(P)$ the set of all components of $[0,1]\setminus P$, and we call these components the \emph{partition intervals}.

\begin{definition}
By $\CMM$ we denote the class of continuous interval maps which are countably Markov and topologically mixing.
\end{definition}

We make two remarks.  First, a map $f\in\CMM$ generally admits many distinct partition sets.  Second, countably infinite Markov partition sets can also be useful for studying some (finitely) piecewise monotone maps.

The basic properties of Markov partitions as regards iteration are summarized in the next lemma.

\begin{lemma}\label{lem:refine}Suppose $f\in \CMM$ with a partition set $P$.
  Then
\begin{enumerate}[label=\textnormal{(\roman*)}]
\item\label{it:mar} $P^n:=\cup_{i=0}^n f^{-i}(P)$ defines a Markov partition for $f^j$, provided $j\leq n+1$.
\item\label{it:int} The partition intervals of $P^n$, $n\geq0$, are the nonempty open intervals of the form
\begin{equation*}
[I_0I_1\cdots I_n]:=I_0\cap f^{-1}I_1 \cap \cdots \cap f^{-n}I_n, \quad \text{where } I_0,\ldots,I_n \in\mathcal{B}(P).
\end{equation*}
\item\label{it:hom} For each $[I_0\cdots I_n]\neq\emptyset$, the restricted map $f^n|_{[I_0\cdots I_n]}:[I_0\cdots I_n] \to [I_n]$ is a homeomorphism.
\item\label{it:forw} The set of accumulation points of $P$ is forward invariant, $f(\textnormal{Acc }P) \subset \textnormal{Acc }P$.
\item\label{it:dense} $Q:=\cup_{n=0}^\infty P^n = \cup_{i=0}^\infty f^{-i}(P)$ is dense in $[0,1]$.
\end{enumerate}
\end{lemma}
\begin{proof} The ideas here are essentially the same as in the theory of finite Markov partitions.  We include the argument for the sake of completeness.
\begin{enumerate}[label=(\roman*)]
\item Clearly $P^n$ is closed, contains $0$ and $1$ and is forward invariant.  Since $f$ is topologically mixing, it must be \emph{strictly} monotone on each of its partition intervals, and therefore the preimage of any singleton is countable.  It follows that $P^n$ is countable. If $U$ is a component of $[0,1]\setminus P^n$, then the sets $U$, $f(U)$, \ldots, $f^{n}(U)$ contain no point from $P$.  Thus for $j\leq n+1$, $f^j|_U$ is a composition of monotone maps, and so is monotone.
\item If $[I_0\cdots I_{n-1}]$ is an interval, then so is $[I_0\cdots I_n]=\left(f^n|_{[I_0\cdots I_{n-1}]}\right)^{-1}(I_n)$, by \ref{it:mar}.  Using induction, we see that all nonempty sets of the form $[I_0\cdots I_n]$ are intervals.  If $x$ belongs to some interval $[I_0\cdots I_n]$, then $f^n(x)$ belongs to $I_n$, hence does not belong to $P$, and therefore $x\notin P^n$.  Thus each interval $[I_0\cdots I_n]$ contains no points from $P^n$, and so is contained in one of the component intervals of $[0,1]\setminus P^n$.  Conversely, if $U$ is a component of $[0,1]\setminus P^n$, then it is contained in the interval $[I_0\cdots I_n]$, where $I_i$ is taken to be the member of $\B(P)$ containing $f^i(U)$, $i=0,\ldots,n$.  The claim follows.
\item By \ref{it:mar} and \ref{it:int}, $f^n$ maps $[I_0\cdots I_{n}]$ monotonically into $I_n$.  It remains to prove surjectivity.  But since $[I_0\cdots I_{n}]$ is a partition interval of $P^n$, its endpoints are in $P^n$ and so their images under $f^n$ are in $P$.
\item Suppose $x\in P$ is a limit of a sequence of points $x_n$ from $P$.  We can choose the sequence $x_n$ using only endpoints of partition intervals and in such a way that $x_{2n}, x_{2n+1}$ are the two endpoints of a common partition interval.  Then the points $f(x_n)$ also belong to $P$ and converge to $f(x)$.  Moreover, topological mixing gives $f(x_{2n})\neq f(x_{2n+1})$.  Therefore infinitely many of the points $f(x_n)$ are distinct from $f(x)$.
\item Fix two distinct partition intervals $I,J$, and let $U$ be any nonempty open set.  By the mixing property we can find a common value of $n$ so that $f^n(U)\cap I\neq\emptyset$ and $f^n(U)\cap J\neq\emptyset$.  Since there is a point of $P$ between $I$ and $J$, it follows by the intermediate value theorem that $U$ contains a point of $P^n$. \hfill \qedhere
\end{enumerate}
\end{proof}

\subsection{Subeigenvectors and Conjugate Maps}\label{subsec:sub}

The purpose of this section is to establish for maps $f\in\CMM$, a close connection between Lipschitz continuous maps conjugate to $f$ and subeigenvalues of the transition matrix of $f$.

\begin{definition}
Let $A$ be a nonnegative matrix with countable index set $S$.  We will say that $\lambda$ is a \emph{subeigenvalue} and that $v$ is a \emph{$\lambda$-subeigenvector} of the matrix $A$ if $v$ is a nonnegative vector satisfying the coordinate-wise inequalities
\begin{equation*}
\sum_{j\in S} A_{ij} v_j \leq \lambda v_i, \quad i\in S.
\end{equation*}
We say $v$ is \emph{deficient in coordinate $i$} if there is strict inequality $\sum_{j\in S} A_{ij} v_j < \lambda v_i$.
\end{definition}

\begin{definition}
The \emph{transition matrix} $A=A(f,P)$ associated to a map $f\in\CMM$ with partition set $P$ is the 0-1 valued finite or countably infinite matrix with rows and columns indexed by the partition intervals and entries
\begin{equation*}
A_{I,J}=\begin{cases}1,&f(I)\supset J \\ 0,&f(I)\cap J=\emptyset\end{cases}.
\end{equation*}
Because of the Markov property (forward-invariance of the partition set), one of these two conditions must hold.

Similarly, the \emph{transition graph} $\Gamma=\Gamma(f,P)$, is the countable directed graph with vertex set $\B(P)$ and with an arrow $I\to J$ if and only if $f(I)\supset J$.
\end{definition}

\begin{remark}
Two remarks are in order.  First, the mixing property of $f$ immediately implies \emph{irreducibility} of $A$ (for all $I,J$ there is $n\in\N$ such that $A^n_{IJ}>0$).  Second, irreducibility of $A$ immediately implies that each subeigenvector has all entries strictly positive.
\end{remark}

\begin{proposition}\label{prop:subeig_nec}
Suppose $f\in\CMM$ with a partition set $P$.  If $f$ admits a conjugate map with some Lipschitz constant $\lambda$, then $A(f,P)$ admits a summable $\lambda$-subeigenvector.
\end{proposition}

\begin{proof}
Suppose $g=\psi\circ f\circ\psi^{-1}$ has Lipschitz constant $\lambda$.  Define $v_I=|\psi(I)|$, $I\in\B(P)$ where vertical bars $|\cdot|$ denote the length of an interval.  Since the intervals $\psi(I)$ are disjoint intervals in $[0,1]$ whose union is cocountable, we get summability $\sum_I v_I=1$.  To see that $v$ is a $\lambda$-subeigenvector, notice that
\begin{equation*}
(Av)_I=\sum_{J\subseteq f(I)} v_J = \sum_{J\subseteq f(I)} |\psi(J)| = |g(\psi(I))| \leq \lambda |\psi(I)| = \lambda v_I
\end{equation*}
where the inequality follows from the Lipschitz property of $g$.
\end{proof}

The next theorem is a partial converse to Proposition \ref{prop:subeig_nec}.  It generalizes a result in~\cite[Theorem 2.5]{B12}, which gives for $\lambda$-eigenvectors (with no deficiency) a conjugate map of constant slope $\pm\lambda$.

\begin{theorem}\label{th:subeig_suf} Suppose $f\in \CMM$ with a partition set $P$.
 If the transition matrix $A(f,P)$ admits a summable $\lambda$-subeigenvector which is deficient in only finitely many coordinates, then $f$ admits a conjugate map with Lipschitz constant $\lambda$.
\end{theorem}

Heuristically, the proof may be summarized as follows.  Our candidate for a conjugate map is a piecewise affine map $g$ which has an identical Markov structure as $f$ but takes the entries of the subeigenvector as the lengths of its partition intervals.  In this way $g$ has constant slope $\pm\lambda$ except on the intervals where the eigenvector was deficient -- and there it has even smaller slope.  Our candidate for the conjugacy $\psi$ is the map that identifies points between the two systems if they have the same itineraries.  By controlling the deficiency of the subeigenvector, we rule out homtervals for $g$ -- this is the essence of equation~\eqref{supinf} below -- which in turn allows us to verify the continuity of $\psi$.  Now we set to the real work of writing down all the details.

\begin{proof}
Denote by $A$ the transition matrix and by $v$ the $\lambda$-subeigenvector.  We may assume $v$ has been scaled so that the sum of its entries is $1$.  We will construct a homeomorphism $\psi$ such that $\psi\circ f\circ\psi^{-1}$ has Lipschitz constant $\lambda$.  We use throughout the proof the notation and results of Lemma~\ref{lem:refine}.  We begin by defining $\psi$ on the sets $P^n$ (and hence on $Q$).    For $[I_0\cdots I_n]\in\B(P^n)$, define
\begin{equation}\label{deltapsi}
\Delta\psi([I_0\cdots I_n]):=
\frac{v_{I_n}}{\prod_{i=0}^{n-1}\lambda_{I_i}}
, \qquad \lambda_{J}:=\frac{(Av)_J}{v_J}\leq\lambda,
\end{equation}
where the empty product (when $n=0$) is taken to be $1$.  This definition has the following two good properties:
\begin{multline}\label{welldefinedDelta}
\sum_{J\subseteq f(I_n)}\Delta\psi([I_0\cdots I_n J]) = \frac{1}{\prod_{i=0}^{n-1} \lambda_{I_i}}\frac{\sum\limits_{J\subseteq f(I_n)} v_J}{\lambda_{I_n}}=\\
=\frac{1}{\prod_{i=0}^{n-1} \lambda_{I_i}}\frac{\left(Av\right)_{I_n}}{\lambda_{I_n}} = \Delta\psi([I_0\cdots I_n]),
\end{multline}
\begin{equation}\label{expansionDelta}
\Delta\psi([I_1\cdots I_n])=\lambda_{I_0}\cdot\Delta\psi([I_0\cdots I_n]).
\end{equation}

Notice that $P^n\subset P^{n+1}$ and that each $[I_0\cdots I_n]\in\B(P^n)$ is subdivided in the partition $P^{n+1}$ into the intervals $[I_0\cdots I_n J]$, where $J$ ranges over all members of $\B(P)$ contained in $f(I_n)$.  It follows by \eqref{welldefinedDelta} there is no ambiguity in the definition
\begin{equation}\label{defpsi}
\psi(x):=\sum_{\emptyset\neq[I_0\cdots I_n]\leq x} \Delta\psi([I_0\cdots I_n]), \quad \textnormal{if }x\in P^n.
\end{equation}
This defines $\psi$ on the set $Q$.  Since $0,1\in P^0$, we find $\psi(0)=0$ (the empty sum) and
\begin{equation*}
\psi(1)=\sum_{[I]\leq 1}\Delta\psi([I])=\sum_{I\in\B(P)}\frac{v_I}{1}=1.
\end{equation*}
Strict monotonicity of $\psi$ on each $P^n$ follows because all entries of $v$ are positive.  It follows also that $\psi$ is strictly monotone on $Q$.

Suppose now that two points $x,x'\in Q$ belong to a common partition interval $I_0$.  Take $n$ minimal so that $x,x'\in P^n$.  Notice that $f$ induces a bijective correspondence between the set of intervals $[I_0\cdots I_n]\in\B(P^n)$ contained in $[x;x']$ (the interval with endpoints $x,x'$) and the set of intervals $[I_1\cdots I_n]\in\B(P^{n-1})$ contained in $[f(x);f(x')]$.  Using \eqref{expansionDelta} and taking sums, we find
\begin{multline}
\label{expansion} |\psi(f(x))-\psi(f(x'))|=\hspace{-.5em}\sum_{\substack{[I_1\cdots I_n] \\ \subset [f(x);f(x')]}}\hspace{-.5em}\Delta\psi([I_1\cdots I_n]) = \sum_{\substack{[I_0\cdots I_n]\\ \subset[x;x']}} \lambda_{I_0} \Delta\psi([I_0\cdots I_n]) = \\
= \lambda_{I_0} |\psi(x)-\psi(x')|, \quad \textnormal{for $x,x'\in Q$ in the same interval $I_0$.}
\end{multline}
Even if $x,x'\in Q$ do not belong to a common partition interval, we still have the intermediate value theorem.  Therefore for each nonempty $[I_1\cdots I_n]\subset[f(x);f(x')]$, there is at least one choice of $I_0$ which yields a nonempty $[I_0\cdots I_n]\subset[x;x']$.  Continuing as in \eqref{expansion}, we obtain the inequality
\begin{equation}
\label{lipschitz} |\psi(f(x))-\psi(f(x'))|\leq\lambda |\psi(x)-\psi(x')|, \quad \textnormal{for arbitrary $x,x'\in Q$}.
\end{equation}
We claim that we can extend the map $\psi:Q\to[0,1]$ to a homeomorphism $\psi':[0,1]\to[0,1]$ by the rule
\begin{equation*}
\psi'(x)=\sup \psi(Q\cap[0,x)),
\end{equation*}
and that the conjugate map $g:=\psi'\circ f \circ \psi'^{-1}$ has Lipschitz constant $\lambda$.  There are several points here to verify, namely,
\begin{enumerate}[label=(\roman*)]
\item\label{it:ext} $\psi'$ is an extension of $\psi$, i.e., $\psi'(x)=\psi(x)$ for $x\in Q$,
\item\label{it:str} $\psi'$ is strictly monotone,
\item\label{it:cts} $\psi'$ is continuous, and
\item\label{it:lip} $g$ has Lipschitz constant $\lambda$.
\end{enumerate}
All four points follow quickly if we can verify the equality
\begin{equation}\label{supinf}
\sup \psi(Q\cap[0,x)) = \inf \psi(Q\cap(x,1]), \qquad x\in[0,1],
\end{equation}
holding to the agreement that the empty set has supremum 0 and infimum 1.  Then point \ref{it:ext} follows from the following observation (using the monotonicity of $\psi$),
\begin{equation}\label{points1and3}
\sup \psi(Q\cap[0,x))=\lim_{\substack{y\to x^-,\\y\in Q}} \psi(y) \leq \psi(x) \leq \lim_{\substack{y\to x^+,\\y\in Q}} \psi(y) = \inf \psi(Q\cap(x,1]).
\end{equation}
Point \ref{it:str} follows from the density of $Q$ and the strict monotonicity of $\psi$.  Point \ref{it:cts} follows from~\ref{it:ext} and~\eqref{points1and3}.  Point \ref{it:lip} follows by replacing each $\psi$ in \eqref{lipschitz} with $\psi'$ and recalling the density of $Q$.

The proof of \eqref{supinf} is quite technical and is deferred to the appendix.
\end{proof}

\subsection{Reverse Salama Entropy}\label{subsec:sal}

In this section we define a new notion of entropy for topological Markov chains, called reverse Salama entropy.  There are two good reasons motivating our definition.  The first is that reverse Salama entropy plays a key role related to summable subeigenvectors of countable matrices - Theorem~\ref{th:sumsubeig}.  The second is that it has a dynamical interpretation for interval maps - Theorem~\ref{th:growthpreimages}.

Let $A$ be a countable matrix%
\footnote{Our interest is in transition matrices of interval maps, $A=A(f,P)$, $S=\B(P)$. But the theory we develop here is valid for general countable state topological Markov chains.}
whose entries are zeros and ones, and whose rows and columns are indexed by some
countable set $S$ (as in states).  Assume that $A$ is \emph{irreducible}, i.e., for each $i,j\in S$ there is $n\geq1$ such that $A^n_{ij}>0$.  We associate to $A$ the directed graph $\Gamma$ with vertex set $S$ and with arrows $i\to j$ if and only if $A_{ij}=1$.  A \emph{path} in $\Gamma$ of \emph{length} $n$ is an $n+1$-tuple of vertices, denoted $[i_0\cdots i_n]$, with the property that there are arrows $i_k\to i_{k+1}$ in $\Gamma$ for each $k=0,\cdots,n-1$; we say that this path \emph{begins} at $i_0$ and \emph{ends} at $i_n$.  The set of all finite paths is denoted $\L$ (as in language).  A \emph{loop} is a path which begins and ends at the same vertex.  Irreducibility of $A$ tells us that $\Gamma$ is \emph{connected}, i.e., for any pair of vertices $i,j$, there exists a path which begins at $i$ and ends at $j$.  An \emph{infinite path} in $\Gamma$ is a sequence of vertices $i_0 i_1 i_2 \ldots \in S^{\N}$ with the property that there are arrows $i_k\to i_{k+1}$ in $\Gamma$ for each $k\geq0$.  The \emph{topological Markov chain} associated with $A$ is the dynamical system $(\Sigma,\sigma)$ where $\Sigma\subset S^{\N}$ consists of all infinite paths in $\Gamma$ and $\sigma:\Sigma\to\Sigma$ is the shift transformation $i_0 i_1 \ldots \mapsto i_1 i_2 \ldots$.  We define a topology on $\Sigma$ by giving $S$ the discrete topology, $S^{\N}$ the product topology, and $\Sigma\subset S^{\N}$ the inherited subspace topology.  Then the irreducibility of $A$ corresponds to the \emph{topological transitivity} of $(\Sigma,\sigma)$.

We wish to relate the subeigenvalues of $A$ to the entropy of $\Sigma$.  However, in the absence of compactness, we have to specify carefully which entropy we have in mind.  We follow the approaches of Gurevich \cite{G} and Salama \cite{Sa}, and define entropy by counting paths in $\Gamma$.
\begin{equation}\label{paths}
\begin{aligned}
p_{ab}^{(n)} &:= (A^n)_{ab} = \# \{ [i_0 \cdots i_n] \in \L : i_0 = a, i_n=b \}, \\
p_{a\dot}^{(n)} &:= \textstyle\sum_{b}(A^n)_{ab} = \# \{ [i_0 \cdots i_n] \in \mathcal{L} : i_0 = a \}, \\
p_{\dot b}^{(n)} &:= \textstyle\sum_{a}(A^n)_{ab} = \# \{ [i_0 \cdots i_n] \in \mathcal{L} : i_n = b \}. \\
\end{aligned}
\end{equation}

We will consider three kinds of entropy.  All three are defined in terms of fixed vertices $a,b$; transitivity implies that they do not depend on a concrete choice of $a,b$.  They are
\begin{equation}
\begin{aligned}
h_{Gur}(\Sigma) &:= \limsup_{n\to\infty} \frac1n \log p_{ab}^{(n)}, && \text{the \emph{Gurevich entropy} of $\Sigma$,} \\
h_{Sal}(\Sigma) &:=\limsup_{n\to\infty} \frac1n \log p_{a\dot}^{(n)}, && \text{the \emph{Salama entropy} of $\Sigma$, and} \\
h_{RevSal}(\Sigma) &:=\limsup_{n\to\infty} \frac1n \log p_{\dot b}^{(n)}, && \text{the \emph{reverse Salama entropy\footnotemark} of $\Sigma$.}
\end{aligned}
\end{equation}
\footnotetext{The reason for the name \emph{reverse Salama entropy} is simple -- it is nothing more than the Salama entropy of the chain we obtain by taking the transpose of $A$, or equivalently, by reversing the direction of all arrows in the graph $\Gamma$.}

We record the following theorem of Bobok and Bruin, which says that the topological entropy of a map from $\CMM$ is given by the Gurevich entropy of the corresponding topological Markov chain.

\begin{theorem}[\cite{BB}]\label{th:Gurnew}
Let $f\in\CMM$ with partition set $P$.  Let $\Sigma$ be the associated topological Markov chain.  Then
\begin{equation*}
h_{Gur}(\Sigma) = h(f).
\end{equation*}
\end{theorem}

Reverse Salama entropy is also useful for studying interval maps%
\footnote{By way of contrast, Salama entropy is not particularly meaningful for interval maps, because different choices of the partition set $P$ (for a fixed map $f$) can yield Markov chains with different Salama entropies.} from the class $\CMM$; it measures the growth rate of the cardinality of iterated preimage sets.

\begin{theorem}\label{th:growthpreimages}
Let $f\in\CMM$ with partition set $P$.  Let $\Sigma$ be the associated topological Markov chain.  Then
\begin{equation*}
h_{RevSal}(\Sigma) = \limsup_{n\to\infty} \frac1n \log \# f^{-n}(\{x\}),
\end{equation*}
for any $x\in[0,1]\setminus\textnormal{Acc}(P)$.
\end{theorem}

\begin{proof}
Fix $x\in[0,1]\setminus\text{Acc}(P)$.  We need to show that the number of iterated preimages of $x$ grows like $h_{RevSal}(\Sigma)$.  We use the sets $P^n$ and the corresponding partition intervals as described in Lemma~\ref{lem:refine}.  We also use the observation that $\text{Acc}(P)$ is invariant.  Therefore $x$ itself and every one of its iterated preimages belongs to the closure of at least one and at most two of the partition intervals $\B(P)$.

Suppose first that $x$ is in one of our (open) partition intervals $J\in\mathcal{B}(P)$.  Fix $n$.  The number $p_{\dot J}^{(n)}$ is the number of nonempty intervals $[I_0\cdots I_n]$ with $I_n=J$.  Each of these intervals is mapped homeomorphically by $f^n$ onto $J$, yielding $p_{\dot J}^{(n)}$ distinct preimages for $x$ under $f^n$.  There are no other preimages: the other partition intervals of $P^n$ have images under $f^n$ disjoint from $J$.  Taking logarithms and sending $n\to\infty$, the result follows.

Now suppose that $x$ is the common endpoint of two consecutive partition intervals $J, K$.  Fix $n$.  We have
\begin{equation*}
\frac12 p_{\dot J}^{(n)} \leq \# f^{-n}(\{x\}) \leq p_{\dot J}^{(n)} + p_{\dot K}^{(n)}.
\end{equation*}
The left-hand inequality is because because each nonempty interval $[I_0\cdots I_n]$ with $I_n=J$ has an endpoint mapped by $f^n$ to $x$, and each such endpoint can belong to at most two such intervals.  The right-hand inequality is because any interval $[I_0\cdots I_n]$ containing an $n$th preimage of $x$ must have $I_n=J$ or else $I_n=K$.  Exponential growth rate is not affected by dividing by 2, and if two sequences have a common exponential growth rate, then so does their sum.

There is one remaining case to consider, when $x$ is the endpoint of exactly one partition interval $J$.  Since by hypothesis, $x\not\in\textnormal{Acc } P$, this can happen only when $x\in\{0,1\}$.  The proof is the same as the previous case, but with the inequality $\frac12 p_{\dot J}^{(n)} \leq \# f^{-n}(\{x\}) \leq p_{\dot J}^{(n)}$.
\end{proof}

\subsection{First Entrance Paths}

As a technical tool, we will need an alternative characterization of reverse Salama entropy in terms of first entrance paths to a fixed vertex.  Thus, we define the path counts
\begin{equation}\label{paths2}
\begin{aligned}
\taboo{p}{c\,}{ab}{(n)} &:= \# \{ [i_0 \cdots i_n] \in \L : i_0 = a, i_n=b, i_j \neq c, 1\leq j\leq n-1 \}, \\
\taboo{p}{c\,}{\dot b}{(n)} &:= \# \{ [i_0 \cdots i_n] \in \L : i_n = b, i_j \neq c, 0\leq j\leq n-1 \}.
\end{aligned}
\end{equation}

\begin{theorem}\label{th:sal}
In a transitive, countable-state, topological Markov chain $\Sigma$,
\begin{equation*}
h_{RevSal}(\Sigma) > h_{Gur}(\Sigma) \implies h_{RevSal}(\Sigma)=\limsup_{n\to\infty} \frac1n \log \taboo{p}{a}{\dot a}{(n)},
\end{equation*}
where $a\in S$ is an arbitrary fixed vertex.
\end{theorem}
\begin{proof}
The inequality $h_{RevSal}(\Sigma) \geq \limsup_{n\to\infty}\frac1n \log \taboo{p}{a}{\dot a}{(n)}$ is obvious from the definitions.  For the reverse inequality, it suffices to show that the radius of convergence of the power series $\sum p_{\dot a}^{(n)} z^n$ is greater than or equal to that of $\sum \, \taboo{p}{a}{\dot a}{(n)}z^n$.

Every length $n$ path $[i_0 \cdots i_{n-1}\, a]$ can be represented uniquely as the concatenation of a length $k$ first entrance path $[i_0 \cdots i_{k-1}\, a]$, $k=\min\{j\in\{0,\ldots,n\} : i_j=a\}$, and a length $n-k$ loop $[a\, i_{k+1} \cdots i_{n-1}\, a]$.  Therefore we obtain the product identity
\begin{equation*}
\underbrace{\sum_{n=0}^\infty p_{\dot a}^{(n)} z^n}_{\substack{\text{convergence radius:}\\ \exp -h_{RevSal}(\Sigma)}}
= \sum_{n=0}^\infty \left( \sum_{k=0}^n \taboo{p}{a\,}{\dot a}{(k)} \cdot p_{aa}^{(n-k)} \right) z^n =
\left( \sum_{n=0}^\infty \taboo{p}{a}{\dot a}{(n)} z^n \right)
\underbrace{\left( \sum_{n=0}^\infty p_{aa}^{(n)} z^n \right)}_{\substack{\text{convergence radius:}\\ \exp -h_{Gur}(\Sigma)}}
\end{equation*}
The radius of convergence of a product of concentric power series is at least the minimum of the radii of the factors.  Since the Reverse Salama entropy was assumed to be strictly larger than the Gurevich entropy, we are finished.
\end{proof}

\begin{remark}
An analogous statement (using last exit paths) holds for Salama entropy.
\end{remark}



\subsection{Subeigenvalues and Reverse Salama Entropy}\label{subsec:subsal}

Now we are ready to relate the subeigenvalues of the irreducible, countable zero-one matrix $A$ to the entropies of the associated topological Markov chain $\Sigma$.  We use the following power series.
\begin{equation*}
P_{ka}(z):=\sum_{n=0}^\infty p_{ka}^{(n)} z^n, \qquad \taboo{P}{a}{ka}{}(z):=\sum_{n=1}^\infty \taboo{p}{a}{ka}{(n)}z^n.
\end{equation*}

\begin{proposition}[Pruitt, \cite{Pr}]\label{prop:pruitt}
If $A$ admits a $\lambda$-subeigenvector, then we have  $\lambda\geq\exp h_{Gur}(\Sigma)$.  Conversely, If $\lambda>\exp h_{Gur}(\Sigma)$, then $A$ admits a $\lambda$-subeigenvector with deficiency in only one coordinate.
\end{proposition}
We include Pruitt's proof of the converse statement, since we will need the construction later.
\begin{proof}
Fix $a\in S$.  Form a vector $v$ with entries $v_k := P_{ka}(\lambda^{-1})$.  For each $i\in S$,
\begin{multline*}
\lambda^{-1} \sum_k A_{ik} v_k = \lambda^{-1} \sum_{n=0}^\infty \sum_k A_{ik} p_{ka}^{(n)} \lambda^{-n}
= \sum_{n=0}^\infty p_{ia}^{(n+1)} \lambda^{-(n+1)} = v_i - \delta_{ia}.
\end{multline*}
Thus $v$ is a $\lambda$-subeigenvector for $A$ with deficiency in only the coordinate $a$.
\end{proof}

\begin{lemma}[Pruitt, \cite{Pr}]\label{lem:pruitt}
If $v$ is a $\lambda$-subeigenvector for $A$, then $v_k \geq v_a \cdot \taboo{P}{a}{ka}{}(\lambda^{-1})$ for all $a,k\in S$.
\end{lemma}


\begin{theorem}\label{th:sumsubeig}
If $A$ admits a summable $\lambda$-subeigenvector, then we have $\lambda\geq\exp h_{RevSal}(\Sigma)$.  Conversely, If $\lambda>\exp h_{RevSal}(\Sigma)$, then $A$ admits a summable $\lambda$-subeigenvector with deficiency in only one coordinate.
\end{theorem}

\begin{proof}
Suppose $A$ admits a summable $\lambda$-subeigenvector $v$.  By Proposition~\ref{prop:pruitt}, $\lambda\geq h_{Gur}(\Sigma)$.  Suppose that $h_{RevSal}(\Sigma)>h_{Gur}(\Sigma)$. Fix $a\in S$. By Lemma~\ref{lem:pruitt} and the summability of $v$,
\begin{equation*}
v_a \sum_{n=1}^\infty \taboo{p}{a}{\dot a}{(n)} \lambda^{-n} = v_a \sum_{k\neq a} \taboo{P}{a}{ka}{}(\lambda^{-1}) \leq v_a \sum_{k\in S} \taboo{P}{a}{ka}{}(\lambda^{-1}) \leq \sum_{k\in S} v_k < \infty.
\end{equation*}
By Theorem~\ref{th:sal}, the coefficients $\taboo{p}{a}{\dot a}{(n)}$ grow like the reverse Salama entropy. It follows that $\lambda \geq \exp h_{RevSal}(\Sigma)$.

Now suppose $\lambda > \exp h_{RevSal}(\Sigma)$. Define $v$ as in the proof of Proposition~\ref{prop:pruitt}. We need only verify the summability of $v$.  We have
\begin{equation*}
\sum_k v_k = \sum_k \sum_n p_{ka}^{(n)} \lambda^{-n} = \sum_n p_{\dot a}^{(n)} \lambda^{-n}.
\end{equation*}
Convergence follows from the condition on $\lambda$.
\end{proof}

\subsection{The Infimum of Lipschitz Constants}\label{subsec:inf}

We are now ready to state our main results for topologically mixing interval maps admitting countable Markov partitions. In our first theorem we prove that in the class of leo maps from $\CMM$ the infimum $\Lambda(f)$ defined in (\ref{mult:inf}) equals the exponential of the topological entropy.   Moreover, for any map from $\CMM$ we give two characterizations of $\Lambda(f)$ -- one extrinsic, in terms of the associated Markov chain, and the other intrinsic, in terms of the growth rate of the number of preimages under $f$ of a (nearly) arbitrary point from the interval.

Recall that a continuous map $f:[0,1]\to [0,1]$ is called {\em leo} ({\em locally eventually onto}) if for every nonempty open set $U$ there is an $n\in{\mathbb N}$ such that $f^n(U)=[0,1]$.

\begin{theorem}\label{th:leo}Let $f\in\CMM$ with a partition set $P$ be leo. Then
\begin{equation*}
\Lambda(f) = \exp h(f).
\end{equation*}
\end{theorem}
\begin{proof}We know from Corollary \ref{cor:lipbound} that $\Lambda(f)\geq \exp h(f)$. Fix $\lambda>\exp h(f)$. Let $\Sigma$ be the associated topological Markov chain for $f$ with $P$. Then by Theorem \ref{th:Gurnew} also $\lambda>\exp h_{Gur}(\Sigma)$ and from the second part of Proposition \ref{prop:pruitt} we obtain that the transition matrix associted to $f$ admits a $\lambda$-subeigenvector with deficiency in only one coordinate. Since the leo property of $f$ implies that any $\lambda$-subeigenvector is summable, Theorem \ref{th:subeig_suf} ensures that $f$ admits a conjugate map with Lipschitz constant $\lambda$.
\end{proof}

\begin{corollary}\label{co:c1leo}Any $C^1$-map $f\in\CMM$ has to be leo, so by Theorem \ref{th:leo} for such a map the equality $\Lambda(f) = \exp h(f)$ holds true.
\end{corollary}
\begin{proof} Suppose $f:[0,1]\to[0,1]$ is topologically mixing, $C^1$-smooth, but not leo.  It is not possible that both endpoints $0,1$ have backwards orbits which enter $(0,1)$, for otherwise the mixing map $f$ would be leo.  Therefore one of the three sets $\{0\}, \{1\}, \{0,1\}$ is backwards invariant.  Suppose without loss of generality that $f^{-1}\{0\}=\{0\}$ (if $\{1\}$ is backwards invariant, then the situation is symmetric, and if $\{0,1\}$ is a backwards invariant two-cycle, then we replace $f$ by $f^2$).

Suppose first that $f$ has a minimum positive fixed point $a$. Then on the interval $(0,a)$ the graph of $f$ must lie either entirely above or entirely below the diagonal.  If it is below, then $[0,a]$ is forward invariant, contradicting the mixing hypothesis.  If it is above, then let $b$ be the minimum value attained by $f$ on the interval $[a,1]$; but then $0<b\leq a$ and $[b,1]$ is forward invariant, again contradicting the mixing hypothesis.

Now suppose that $f$ has a whole sequence of fixed points $a_i\searrow 0$.  Then $f'(0)=1$.  However, since the interval $[a_{i+1},a_i]$ cannot be forward-invariant, it follows that there is inside this interval a point $x_i$ with $f(x_i)$ outside this interval.  By the mean value theorem, $f'$ takes both positive and negative values on $[a_{i+1},a_i]$.  Since this is for each $i$, continuity of the derivative gives $f'(0)=0$, a contradiction.
\end{proof}

\begin{theorem}\label{th:preimages}Suppose $f\in \CMM$ with a partition set $P$.
  Let $\Sigma$ be the associated topological Markov chain.  Then
\begin{equation*}
\log \Lambda(f) = h_{RevSal}(\Sigma) = \limsup_{n\to\infty} \frac1n \log \# f^{-n}(\{x\}),
\end{equation*}
for any $x\in[0,1]\setminus\textnormal{Acc}(P)$.
\end{theorem}

\begin{proof}
The second equality is Theorem~\ref{th:growthpreimages}. The first equality follows from Proposition~\ref{prop:subeig_nec}, Theorem~\ref{th:subeig_suf}, and Theorem~\ref{th:sumsubeig}, as shown in the diagram below:

\noindent\begin{minipage}{\textwidth}
\begin{minipage}[c][9em][c]{.28\textwidth}\centering
\vspace{0.5em}
$\exists$ a conjugate map \\ with Lipschitz \\ constant $\lambda$
\end{minipage}%
\begin{minipage}[c][9em][c]{.08\textwidth}
\vspace{-.8em}
\rotatebox{20}{$\xRightarrow{\text{Pr~\ref{prop:subeig_nec}}}$} \\[1.5em]
\rotatebox{-20}{$\xLeftarrow{\text{Th~\ref{th:subeig_suf}}}$}
\end{minipage}%
\begin{minipage}[c][9em][c]{.28\textwidth}\centering
\vspace{.7em}
$\exists$ a summable $\lambda$-subeigenvector. \\[1.5em]
$\exists$ a summable $\lambda$-subeigenvector,
finitely many deficiencies.\\
\end{minipage}%
\begin{minipage}[c][9em][c]{.08\textwidth}\centering
\vspace{-.5em}
$\xRightarrow{\text{Th~\ref{th:sumsubeig}}}$ \\[3.8em]
$\xLeftarrow{\text{Th~\ref{th:sumsubeig}}}$
\end{minipage}%
\begin{minipage}[c][9em][c]{.28\textwidth}\centering
$\log \lambda \geq$ reverse \\
Salama entropy.\\[2.8em]
$\log \lambda >$ reverse \\
Salama entropy.\\
\end{minipage}%
\end{minipage}

\end{proof}

We cannot help but notice that the class of maps treated in our Theorem~\ref{th:preimages} overlaps in part with the class treated by Parry~\cite{P} (transitive, piecewise monotone, not necessarily Markov), for which $\Lambda(f)=\exp h(f)$ (Parry gave a conjugate map with constant slope, $\lambda = \pm \exp h(f)$).  But this is not a problem, because for the maps Parry considers, topological entropy coincides with the growth rate of the number of preimages of a point.  This is the content of a recent theorem by Misiurewicz and Rodrigues (proved in~\cite{MR} for topologically mixing maps; but a trivial modification of the proof gives the strengthened result for transitive maps.)

\begin{theorem}[\cite{MR}]\label{th:misrod}
Suppose $f:[0,1]\to[0,1]$ is topologically mixing and (finitely) piecewise monotone.  Then $h(f)=\lim_{n\to\infty} \frac1n \log \#f^{-n}(x)$ for any $x\in[0,1]$.
\end{theorem}

As a consequence, we can view our Theorem~\ref{th:preimages} as a natural analogue to Parry's result. 
The penalty we pay for allowing countable Markov partitions is that the limit is replaced by a lim sup, and the point $x$ is no longer completely arbitrary.

\subsection{An Example}\label{subsec:ex}

We reconsider an interval map $f$ from~\cite{MRo}, where it was proven that $f$ is not conjugate to an interval map of constant slope, i.e., piecewise linear with each piece sharing a common absolute value of slope. Since we cannot achieve constant slope, it is natural to search instead for conjugate maps with Lipschitz constants as small as possible.

\begin{example}\label{ex:gap}
Let $F:\R\to\R$ be the piecewise linear map with turning points $z\mapsto z-1$ and $z+\frac35\mapsto z+2$ for $z\in\Z$. Then choose a homeomorphism $h:\R\to(0,1)$ and define $f:=h\circ F\circ h^{-1}$ with additional fixed points at $0,1$. Take $P=h(\Z)\cup h(\Z+\frac35)$. Clearly $f$ is countably Markov with partition set $P$. Moreover, $f$ is topologically mixing as a consequence of  \cite[Theorem~5.4]{MRo}.

Every point of the interval, except the endpoints, has exactly 5 preimages, and so by Theorem~\ref{th:preimages} we see that $\Lambda(f)=5$.

Now consider the transition graph $\Gamma(f,P)$. The vertices are the partition intervals, which we label by the rule $I_n=(h(n),h(n+\frac35))$, $J_n=(h(n+\frac35),h(n+1))$, $n\in\Z$.  Figure~\ref{fig:trans} pictures $f$ (with $P$ superimposed) as well as $\Gamma$.

We can simplify our calculations by noticing that there is a topological conjugacy between the vertex shift of $\Gamma$ and the edge shift (in the terminology of~\cite[\S\S2.2-3]{LM}) of the simpler graph $\Gamma'$ (also pictured in Figure~\ref{fig:trans}) corresponding to the matrix $A'$ with entries $A'_{n,n+1}=A'_{n,n}=2$, $A'_{n,n-1}=1$, $n\in\Z$.  If we use the labeling as shown in Figure~\ref{fig:trans}, then this conjugacy is given explicitly by the formula $\phi(x_0x_1\cdots)=(x_0,n_1-n_0)(x_1,n_2-n_1)\cdots$, where $n_i$ is the integer such that $x_i\in\{I_{n_i},J_{n_i}\}$.

\begin{figure}[htb!]
\begin{minipage}[t][8em][c]{.30\textwidth}\centering
\includegraphics[width=.8\textwidth]{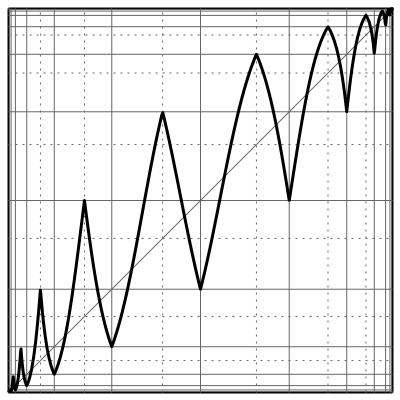}
\end{minipage}%
\begin{minipage}[t][8em][c]{.30\textwidth}\centering
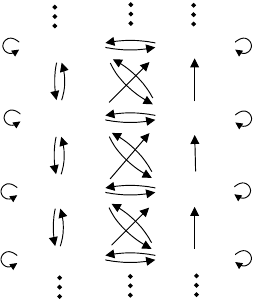
\end{minipage}%
\begin{minipage}[t][8em][c]{.20\textwidth}\centering
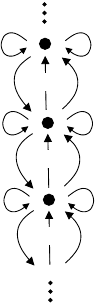
\end{minipage}
\label{fig:trans}\caption{The map $f$, its graph $\Gamma$, and a simpler graph $\Gamma'$ whose edge shift is conjugate to the vertex shift of $\Gamma$.}
\end{figure}

Let $\lambda_{A'}=\limsup_{n\to\infty}\left(A'^n\right)_{0,0}$ denote the Perron value of $A'$.  We have
\begin{equation*}
h(f)=h_{Gur}\left(\text{Vertex shift of }\Gamma\right)=h_{Gur}\left(\text{Edge shift of }\Gamma'\right)=\log \lambda_{A'}.
\end{equation*}
The first equality is just Theorem~\ref{th:Gurnew}.  The second equality is because Gurevich entropy is an invariant of topological conjugacy.  The third equality is a general fact about shift spaces represented by nonnegative matrices.

Now we apply a method from~\cite{BB} to compute the Perron value $\lambda_{A'}$.  Since each row of $A'$ has only finitely many nonzero entries, this number is characterized by the property that $A'$ has a nonnegative $\lambda$-eigenvector if and only if $\lambda\geq \lambda_{A'}$ (This is~\cite[Corollary 1]{BB}, a corollary of~\cite[Theorem 2]{Pr}).  The eigenvectors of $A'$ can be found by solving the linear difference equation $v_{i-1}+2v_{i}+2v_{i+1}=\lambda v_i$, whose characteristic polynomial is $m^2+(2-\lambda)m+1=0$.  The difference equation admits a nonnegative solution if and only if the characteristic polynomial has at least one positive real root, which happens if and only if $\lambda\geq2+2\sqrt2$.  Thus $f$ has topological entropy $h(f)=\log(2+2\sqrt2)\approx\log(4.828)$, and there is a gap $\Lambda(f)>\exp h(f)$.
\end{example}

\begin{remark}A two parameter family of (transient) maps from $\CMM$ for which $\Lambda>\exp h$ can be constructed with the help of \cite[Proposition 14(b)]{BB}.
\end{remark}

\section{Appendix}

We finish the proof of Theorem~\ref{th:subeig_suf} by supplying the justification for equation~\eqref{supinf}.

\begin{proof}
Suppose first that $x\notin Q$.  Then each point $f^n(x)$ belongs to a unique partition interval $I_n\in\B(P)$.  (The sequence $I_0 I_1 \cdots$ is usually called the \emph{itinerary} of $x$).  We are interested in the nested sequence of intervals $[I_0I_1\cdots I_n]$, $n\in\N$, each of which contains $x$.  By the density of $Q$, the endpoints of these intervals are converging to $x$.  Therefore,
\begin{equation*}
\inf \psi(Q\cap(x,1]) - \sup \psi(Q\cap[0,x)) = \lim_{n\to\infty} \Delta\psi([I_0\cdots I_n]).
\end{equation*}
The terms in the limit are monotone decreasing by the monotonicity of $\psi$ -- we must show that they decrease to zero.  Suppose first that some symbol $J$ occurs infinitely often in the itinerary of $x$.  If $f(J)$ is a single partition interval, then it also occurs infinitely often in the itinerary of $x$, and we may replace $J$ by $f(J)$.  Since the mixing hypothesis does not allow for a cycle of partition intervals, we may conclude (after making finitely many such replacements) that $f(J)$ contains more than one partition interval.   Among all the partition intervals contained in $f(J)$ there must exist some $L$ with $\frac{v_L}{\lambda_J v_J}$ maximal; call this ratio $c$ and notice that $c<1$.  We have
\begin{equation*}
\Delta\psi([I_0 \cdots J K]) = \frac{v_K}{\lambda_J v_J} \Delta\psi([I_0 \cdots J]) \leq c \Delta\psi([I_0 \cdots J]).
\end{equation*}
Therefore our decreasing sequence shrinks by the factor $c$ or better infinitely many times, and hence converges to zero.  Suppose now that each symbol in the itinerary of $x$ occurs only finitely often.  By hypothesis, we have $\lambda_J=\lambda>1$ for all but finitely many intervals $J\in\B(P)$.  Since these intervals can occur only finitely many times in the itinerary, we find that the denominator in the expression for $\Delta\psi([I_0 \cdots I_n])$ in \eqref{deltapsi} is eventually monotone increasing with respect to $n$, growing by a factor of $\lambda$ at each step.  Moreover, the numerator in this expression is always bounded by $1$.  Therefore the limit is zero, as desired.

Now suppose $x\in Q$.  We will show that $\psi(x)=\inf \psi(Q\cap(x,1])$.  The proof that $\psi(x)=\sup \psi(Q\cap[0,x))$ is similar.  There are two possibilities to consider.  Either for each $n\geq0$ the set $P^n\cap(x,1]$ has a minimum element $x_n$, or else for some $n_0$, the set $P^{n_0}\cap(x,1]$ accumulates at $x$.  In the first case we obtain an ``itinerary'' $I_0I_1\cdots$ defined by the property that for each $n$, $[I_0\cdots I_n]$ is the component $(x,x_n)$ of $[0,1]\setminus P^n$.  Then $\inf \psi(Q\cap(x,1])$ is given by the monotone decreasing limit $\lim \Delta\psi([I_0\cdots I_n])$ and we proceed as before.  In the second case we obtain a whole sequence of points $y_i\in P^{n_0}$ which converge monotonically $y_i \searrow x$.  We may apply the definition \eqref{defpsi} to calculate
\begin{multline*}
\inf \psi(Q\cap(x,1]) - \psi(x) = \lim_{i\to\infty} \psi(y_i)-\psi(x) = \lim_{i\to\infty} \sum_{\substack{[I_0\cdots I_{n_0}]\neq\emptyset \\ x\leq[I_0\cdots I_{n_0}]\leq y_i}} \Delta\psi([I_0\cdots I_{n_0}]) \\
= \lim_{i\to\infty} \sum_{j=i}^\infty \sum_{\substack{[I_0\cdots I_{n_0}]\neq\emptyset \\ y_{j+1} \leq [I_0\cdots I_{n_0}] \leq y_j}} \Delta\psi([I_0\cdots I_{n_0}])=0.
\end{multline*}
The rearrangement of the sum is justified because for each nonempty $[I_0\cdots I_{n_0}]$ between $x$ and $y_i$ there is exactly one $j\geq i$ such that $y_{j+1}\leq [I_0\cdots I_{n_0}]\leq y_j$, and because a convergent series of nonnegative terms may be rearranged at will.  The limit is zero because it is the limit of the tail sums of a convergent series.

\end{proof}


\begin{thebibliography}{9}

\bibitem{ALM}
Ll.~Alsed{\`a}, J.~Llibre and M.~Misiurewicz,
\emph{Combinatorial Dynamics and Entropy in Dimension One},
2nd edition, Advanced Series in Nonlinear Dynamics~\textbf{5},
World Scientific, Singapore, 2000.

\bibitem{B}
J.~Bobok,
On Lipschitz extension of interval maps.
\emph{Qual. Theory Dyn. Syst.} \textbf{4}~(2003), no. 1, 17--30.

\bibitem{B12}
J.~Bobok
Semiconjugacy to a Map of a Constant Slope.
\emph{Studia Math.} \textbf{208}~(2012), no. 3, 213--228.

\bibitem{Bo05}
J.~Bobok,
The topological entropy versus level sets for interval maps (part II).
\emph{Studia Mathematica} \textbf{166}~(2005), no. 1, 11--27.
\bibitem{BB}
J.~Bobok and H.~Bruin,
Constant Slope Maps and the Vere-Jones Classification.
\emph{Entropy} \textbf{18} (2016), no. 6, Paper No. 234, 27 pp.

\bibitem{Ce1}
L.~Cesari, \emph{Variation, multiplicity, and semicontinuiuty}, Amer. Math.
  Monthly \textbf{65} (1958), 317--332.

\bibitem{G}
B.~Gurevi\v{c},
Topological entropy of a countable Markov chain. (Russian)
\emph{Dokl. Akad. Nauk SSSR} \textbf{187}~(1969), 715--718.

\bibitem{KH}
A.~Katok and B.~Hasselblatt,
\emph{Introduction to the modern theory of dynamical systems},
Cambridge University Press, Cambridge, 1995.

\bibitem{LM}
D.~Lind and B.~Marcus,
\emph{An Introduction to Symbolic Dynamics and Coding},
Cambridge University Press, Cambridge, 1995.

\bibitem{MR}
M.~Misiurewicz and A.~Rodrigues,
Counting preimages.
\emph{Ergod. Th. \& Dynam. Sys.}  Published online 24 January, 2017. doi: 10.1017/etds.2016.103. 20 pages.

\bibitem{MRo}
M.~Misiurewicz and S.~Roth,
No semiconjugacy to a map of constant slope.
\emph{Ergod. Th. \& Dynam. Sys.}  \textbf{36}~(2016), no. 3, 875--889.

\bibitem{MS}
M.~Misiurewicz and W.~Szlenk,
Entropy of piecewise monotone mappings.
\emph{Studia Math.} \textbf{67}~(1980), no. 1, 45--63.

\bibitem{MT}
J.~Milnor and W.~Thurston,
On iterated maps of the interval.
\emph{Dynamical Systems (Lecture Notes in Mathematics, 1342).}
Springer, Berlin, 1988, pp. 465-563.


\bibitem{P}
W.~Parry,
Symbolic dynamics and transformations of the unit interval.
\emph{Trans. Amer. Math. Soc.} \textbf{122} (1966), 368--378.

\bibitem{Pr}
W.~Pruitt,
Eigenvalues of non-negative matrices.
\emph{Ann. Math. Statist.} \textbf{35} (1964), 1979--1800.

\bibitem{Sa}
I.~Salama,
Topological entropy and recurrence of countable chains.
\emph{Pacific J. Math.} \textbf{134} (1988), no. 2, 325--341.

\bibitem{Y}
Y. Yomdin,
Volume growth and entropy.
\emph{Israel J. Math.} \textbf{57} (1987), no. 3, 285--300.

\end{thebibliography}
\end{document}